\documentclass[11pt]{amsart}
\usepackage{mathtools}
\usepackage{amssymb}
\usepackage{amsthm}
\usepackage{graphicx}
\usepackage{capt-of}
\usepackage[textwidth=6in,textheight=9.5in,centering]{geometry}

\newtheorem{thm}{Theorem}[section]

\newtheorem{lem}[thm]{Lemma}
\newtheorem{cor}[thm]{Corollary}
\newtheorem{heuristic}[thm]{Heuristic}

\newtheorem{conj}[thm]{Conjecture} 
\theoremstyle{definition}
\newtheorem{definition}[thm]{Definition}

\numberwithin{equation}{section}
\newcommand{\F}{\mathbb{F}}

\newcommand{\N}{\mathbb{N}}  
     
\newcommand{\ff}[1]{\mathbb{F}_{#1}} 
\newcommand{\set}[2]{\{#1 \: : \: #2\}}
\newcommand{\lv}{\lvert}
\newcommand{\rv}{\rvert}

\title{Estimating the Number Of Roots of Trinomials over Finite Fields}
\author{Zander Kelley$^\star$, Sean W. Owen$^\star$}\thanks{$^\star$Partially 
supported by NSF grants DMS-1156589, DMS-1460766, and CCF-140902.} 
\begin{document}
\maketitle

\begin{abstract}
We show that univariate trinomials $x^n + ax^s + b \in \mathbb{F}_q[x]$ can have at most $\delta \Big\lfloor \frac{1}{2} +\sqrt{\frac{q-1}{\delta}} \Big\rfloor$ distinct roots in $\mathbb{F}_q$, where $\delta = \gcd(n, s, q - 1)$. 
We also derive explicit trinomials having $\sqrt{q}$ roots in $\mathbb{F}_q$ when $q$ is square and $\delta=1$, thus showing that our bound is tight for an infinite family of finite fields and trinomials. 
Furthermore, we present the results of a large-scale computation which suggest that an $O(\delta \log q)$ upper bound may be possible for the special case where $q$ is prime. 
Finally, we give a conjecture (along with some accompanying computational and theoretical support) that, if true, would imply such a bound.
\end{abstract}

\section{Introduction} 
For univariate polynomial equations defined over a field, it is desirable to obtain general upper bounds on the number of solutions given in simple terms of plainly available information, such as the coefficients, exponents, or number of terms. 
The ubiquitous example of this is the degree bound, but over non-algebraically closed fields, it is possible to considerably improve upon the degree bound for certain non-negligible families of polynomials.
Over the real numbers, Descartes' Rule of Signs implies that a $t$-nomial $f$ must have less than $2t$ real roots.
For sparse polynomials - those with a small number of nonzero terms - \linebreak this can provide a remarkable improvement on the trivial upper estimate given by the degree of $f$.

In \cite{canetti}, the authors establish a finite field analogue of Descartes' Rule: a sparsity-dependent upper
bound on the number of roots of a $t$-nomial over $\mathbb{F}_q$.
More recently, an improved upper bound was derived in \cite{sparse_polys}.
Here, we investigate possible further improvements to the bound for the special case of $t = 3$. 
This can be considered the smallest nontrivial choice of $t$, since the zero sets of univariate binomials are easily characterized - they are simply cosets of subgroups of $\F_q^*$, possibly together with $0 \in \F_q$.

\begin{thm} \label{sparse}
\cite[Theorems 2.2 and 2.3]{sparse_polys} 
Let $$f(x) = c_1 x^{a_1} + c_2 x^{a_2} + \cdots + c_t x^{a_t} \in \F_q[x]$$ with all $c_i$ nonzero and $a_1 > a_2 > \cdots > a_t = 0$.
If $f$ vanishes on an entire coset of a subgroup $H \subseteq \F_q^*$, then
$$
 \#H \in \set{k \in \N}{\textnormal{for each } a_i \textnormal{, there is an }a_j \textnormal{ with }  j \neq i  \textnormal{ and }  a_i \equiv a_j \: (\bmod \: k)}.
$$

\noindent
Furthermore, let $R(f)$ denote the number of distinct roots of $f$ in $\F_q$, and suppose $R(f) > 0$.
If $C$ denotes the maximal cardinality of a coset on which $f$ vanishes, then
$$ R(f) \leq 2 (q-1)^{1-1/(t-1)} C^{1/(t-1)}. $$
\end{thm}

\pagebreak

For a trinomial $f(x) = x^n + ax^s + b \in \F_q[x]$, with $a$ and $b$ nonzero, associate the parameter $$\delta = \gcd(n,s,q-1).$$
Suppose that $R(f) > 0$.
It follows from Theorem \ref{sparse} that if $f$ vanishes on a coset of size $C$, then $n \equiv s \equiv 0 \: (\bmod \: C)$. Since $C$ must divide $\#\F_q^*$, we have that $C$ divides $\delta$.
On the other hand, if $f$ vanishes at $\alpha \in \F_q$, then $\alpha \in \F_q^*$, and $f$ vanishes on the entire coset $\set{x \in \F_q^*}{x^\delta = \alpha^\delta}$ of order $\delta$.
So, in the trinomial case we have explicitly that $C = \delta$, and the bound given above simplifies to 
$$R(f) \leq 2 \sqrt{\delta (q-1)}.$$
As pointed out in \cite{rojas2}, 
    this bound for trinomials is also a consequence of an earlier result from \cite{rojas1} which bounds the number of cosets $S_i \subset \F_q^*$ needed to express the zero set of a sparse polynomial as a union of the form $ \bigcup_{i=1}^N S_i$.
Our first result refines this upper bound.

\begin{thm} \label{upper_bound}
The roots of a trinomial
$$f(x) =  x^{n} + ax^{s} + b \in \F_q[x]$$
are the union of no more than $\Big\lfloor \frac{1}{2}+\sqrt{\frac{q-1}{\delta}} \Big\rfloor$ cosets of the subgroup $H \subseteq \F_q^*$ of size $\delta$.
\end{thm}

Consequently, we now have $R(f) \leq \delta \Big\lfloor \frac{1}{2} +\sqrt{\frac{q-1}{\delta}} \Big\rfloor$, improving the previous result by approximately a factor of $2$ when $\delta \ll q$. 
The method of proof is elementary but interesting: given a trinomial with $\delta = 1$ and $r$ roots in a field of undetermined size, we construct $r^2 - r + 1$ distinct nonzero elements in the field, giving a lower bound on its size.

Additionally, we show that when $\delta=1$, this new bound is optimal for even-degree extensions of $\F_p$.
If $q$ is an even power of a prime $p$ and $\delta=1$, the bound reduces to $R(f) \leq \sqrt{q}$, and we can indeed construct trinomials with $\delta=1$ and $\sqrt{q}$ distinct roots in $\F_q$.   

\begin{thm} \label{june_trinomials}
For any odd prime $p$, the trinomial $x^{p^k} + x - 2$ 
has exactly $p^k$ roots in $\F_{p^{2k}}$.
\end{thm}

We prove Theorem \ref{june_trinomials} via linear-algebraic techniques: the extremal examples provided are translations of linear maps with null-spaces of exactly half the dimension of $\ff{q}$ as a vector space over $\ff{\sqrt{q}}$.
The optimality of the bound is somewhat murkier when $\F_q$ is not an even-degree extension. 
Trinomials with nearly as many roots have been found for some other cases; for example, when $q$ is a cube, the authors of \cite{rojas2} give the example $f(x) = x^{1+q^{1/3}} + x + 1$ which has $q^{1/3} + 1$ roots.

Most notably, the question of optimality of the bound remains open for the prime field case.
We remark that out of all examples of which we are aware, including those given in \cite{rojas2}, the only sparse polynomials which vanish at a substantial number of points do so by abusing some obvious algebraic structure of $\F_q$ - they either vanish on an entire translation of a subspace or on an entire coset of a nontrivial subgroup.
Trinomials over prime fields which have $\delta = 1$ are deprived of both of these luxuries, and accordingly, finding examples with many roots seems to be difficult.

Let $R_p$ denote the maximum value of $R(f)$ over all trinomials in $\F_p[x]$ having $\delta = 1$. 
Recall that by Fermat's little theorem, if $\tilde{n} \equiv n \: (\bmod \: p - 1)$ and $\tilde{s}  \equiv s \: (\bmod \: p - 1)$, then the two polynomials $f(x) = x^n + ax^s + b$ and $\tilde{f}(x) = x^{\tilde{n}} + ax^{\tilde{s}} + b$ define the same mapping on $\F_p^*$, so it is possible to compute $R_p$ in straightforward way by enumerating trinomials with degree less than $p - 1$ and counting their roots.
In \cite{rojas2}, $R_p$ is computed for all primes up to 16633, and they find no instances in which $R_p$ exceeds $2 \log p$.
As a result of a large-scale computation, we observe that the inequality $R_p \leq 2 \log p$ continues to hold for all primes up to 139571.
Therefore, it appears that the current bound, $R_p = O(\sqrt{p})$, is still far from optimal for trinomials over $\F_p$, but we have been unsuccessful in proving any stronger version of Theorem \ref{upper_bound} for prime fields.

It is known that if $f$ is allowed to range over all polynomials in $\F_p[x]$, the distribution of $R(f)$ approaches a Poisson distribution with mean $1$ as $p \to \infty$  \cite{random}.
That is, the proportion of $f \in \F_p[x]$ with $R(f) = r$ is approximately  $e^{-1}/r!$ when $p$ is sufficiently large. 
Based on computational experiments, this also appears to be true when $f$ ranges over just the set of trinomials in $\F_p[x]$ with $\delta=1$. 
This is certainly \textit{not} the case if $f$ were to range over, for example, the set of \textit{all} trinomials, or the set of {\em tetranomials} with $\delta=1$ due to the presence of $f$ which vanish on large cosets. 
On the other hand, the set of trinomials with $\delta=1$ appears to behave similarly to what we would expect from an $f$ randomly selected from all of $\F_p[x]$. Apparently, restriction of $f \in \F_p[x]$ to trinomial with $\delta = 1$ provides very little statistical information about $R(f)$.

\begin{heuristic} \label{behave}
With respect to root number, the set of trinomials $f \in \F_p[x]$ with $\delta=1$ behaves like a uniform random sample of polynomials from $\F_p[x]$. 
That is, when $p$ is large enough, the values of $R(f)$ behave like they are given by random variables with distribution function $\rho(r) = e^{-1}/r!$ (a Poisson distribution with mean 1).
\end{heuristic}

In Section 4, we show that Heuristic \ref{behave} allows us to make fairly accurate guesses of the actual values of $R_p$ recorded by our computations.
Therefore, it may be that the observed logarithmic growth of $R_p$ is not due to any special property of trinomials with $\delta = 1$, but rather emerges as a statistical consequence of this set being so ``ordinary," together with the exponential decay of the Poisson distribution.
We phrase this formally as the following conjecture that the distributions of $R(F)$ and $R(f)$, with $F$ ranging over $\F_p[x]$ and $f$ ranging over trinomials with $\delta = 1$, differ by at most a constant factor.

\begin{conj} \label{con}
Define
\begin{align*}
M_p &= \set{f \in \F_p[x]}{ \deg f < p},  \\
T_p &= \set{x^n + ax^s + b}{ (a,b) \in (\F_p^*)^2, \; 0 < s < n < p - 1, \; \textnormal{and }\gcd(n,s,p-1) = 1}.
\end{align*}
Let $\mu(p,r)$ denote the proportion of $f \in M_p$ with $R(f) = r$ and $t(p,r)$ denote the proportion of $f \in T_p$ with $R(f) = r$.
There exists a constant $\lambda \in \mathbb{R}$ such that

$$ t(p,r) \leq \lambda  \mu(p,r) ,$$
for all $p$ prime and $r \in \N$.

\end{conj}

If these two distributions do in fact differ by at most a constant factor, then we are able to readily derive the logarithmic upper bound for $R_p$ suggested by our experiments. 
And, in turn, we could extend such a bound to trinomials with $\delta > 1$ by noticing that $x^n + ax^s + b$ has at most $\delta$ roots for every root of $x^{n / \delta} + ax^{s / \delta} + b$.

\begin{cor} \label{cor}
Suppose Conjecture \ref{con} is true. Then, we have the asymptotic bound
$$ R_p = \max \set{R(f)}{f \in T_p} = O \left( \frac{\log p}{\log \log p} \right). $$
\end{cor}

\begin{proof}

Let $M_p(r) = \set{f \in M_p}{R(f) = r}$ and $T_p(r) = \set{f \in T_p}{R(f) = r}$ so that 
$$\mu(p,r) = \frac{\#M_p(r)}{\#M_p} \textnormal{ and } t(p,r) = \frac{\#T_p(r)}{\#T_p}.$$
\pagebreak

\noindent
We can bound $\#M_p(r)$ from above by counting polynomials of the form
$$ \left( \prod_{i=1}^r (x - \alpha_i) \right) \left( \sum_{i=0}^{p-1-r} c_i x^i \right), $$
with $\alpha_i \in \F_p$ distinct, which gives
$$ \mu(p,r) = \frac{\#M_p(r)}{\#M_p} =  \frac{\#M_p(r)}{p^{p}} \leq \frac{\binom{p}{r} p^{p-r}}{p^{p}} =  \binom{p}{r} \frac{1}{p^r}  \leq \frac{1}{r!}.$$
Obviously $\#T_p \leq p^4$, so assuming the existence of $\lambda$ defined in Conjecture \ref{con}, we have
$$ \#T_p(r) \leq \lambda  \mu(r,p) \#T_p \leq \frac{\lambda \#T_p}{r!} \leq \frac{\lambda p^4}{r!} .$$
If $\lambda p^4 / r! < 1$ then the set $T_p(r)$ is empty, so we must have $\lambda p^4 / R_p! \geq 1$, or equivalently, 
$ \log (R_p!) \leq \log (\lambda p^4 ).$ By applying Stirling's approximation, we get the asymptotic bound
$$ R_p \log R_p \sim \log (R_p!) \leq \log (\lambda p^4 ) = 4 \log p + \log \lambda = O(\log p). $$
By considering the growth order of the inverse function of $y = x \log x$, we obtain $$R_p = O \left( \frac{\log p}{\log \log p} \right).$$

\end{proof}

We remark that in \cite{rojas2}, it is shown that under the Generalized Riemann Hypothesis, there exists an infinite sequence of primes 
$\left( p_k \right)_{k=1}^\infty$ satisfying the lower bound $R_{p_k} = \Omega \left( \frac{\log {p_k}}{\log \log {p_k}} \right)$. So, the truth of both Conjecture \ref{con} and GRH would imply that the bound
 in Corollary \ref{cor} is, up to a multiplicative constant, asymptotically optimal.

Finally, we prove the following theoretical result which states that Conjecture \ref{con} is true if we consider only trinomials of bounded degree as we take $p$ to infinity.
This weaker result is suggestive but certainly not sufficient to imply the bound in Corollary \ref{cor}.
In particular, Theorem \ref{poisson_density} shows the existence of $\lambda_N \in \mathbb{R}$ such that $t_N(p,r) \leq \lambda_N \mu(p,r)$, but we do not have a bound on the set $\set{\lambda_N}{N \in \N}$.

\begin{thm} \label{poisson_density} 
Suppose $n,s \in \N$ with $0 < s < n$ and $\gcd(n,s) = 1$.
As $p \to \infty$, the proportion of pairs $(a,b) \in (\F_p^*)^2$, such that $f(x) = x^n + ax^s + b$ has $R(f) = r$, converges to
$$
\begin{dcases}
\frac{ \big[e^{-1} (n-r)! \big] }{\hfill r! (n-r)! \ } & \mbox{if } r < n \\
\: 1/r! &\mbox{if } r = n,
\end{dcases}
$$
where $[\cdot]$ denotes the ``nearest integer" function.\\

\noindent
Furthermore, fix $N \in \N$, and let
$$T_{p,N} = \set{x^n + ax^s + b}{(a,b) \in (\F_p^*)^2, \; 0 < s < n \leq N, \; \textnormal{and } \gcd(n,s,p-1) = 1}.$$ 
Let $\mu(p, r)$ be defined  as in Conjecture \ref{con}, 
and let $t_N(p,r)$ denote the proportion of $f \in T_{p,N}$ with $R(f) = r$.
We then have

$$ \limsup_{p \to \infty} \left( \max_{r \leq N} \frac{ t_N(p,r)}{\mu(p,r)}   \right) \leq e .$$

\end{thm}

\section{New Upper Bound and Extremal Trinomials}

\begin{definition}
For $n, s$ fixed, define the family of trinomials in $\ff{q}[x]$ 
$$C(n, s) = \set{ f_c(x) = cx^{n} - (c+1)x^{s} + 1}{c\neq -1,0}. $$
\end{definition}

Observe that $C(n, s)$ is exactly the set of trinomials with
\begin{itemize}
\item support $\{ n, s, 0 \}$ 
\item constant term $1$
\item $f(1) = 0$ 
\end{itemize}
This is clear because $f(1) = 0$ if and only if $f$'s coefficients sum to zero.
We introduce this family of trinomials because they have the following useful property.

\begin{lem} \label{disjoint_roots}
Let $G \subseteq \F_q^*$ be the unique multiplicative subgroup of order $N$, and suppose that $\gcd(n, s, N) = 1$. 
The only root in $G$ shared by any two members of $C(n, s)$ is $\alpha = 1$.
\end{lem}

\begin{proof}
$f_c(\alpha) = 0$ is equivalent to the following linear equation in $c$: 
$$c(\alpha^{n} - \alpha^{s}) = \alpha^{s} - 1.$$ 
This has multiple solutions in $c$ if and only if both $\alpha^{n} - \alpha^{s} = 0$ and $\alpha^{s} - 1 = 0$.
Since $G$ is a cyclic group of order $N$ and $\gcd(n, s, N) = 1$, the only $\alpha \in G$ such that $\alpha^{n} = \alpha^{s} = 1$ is $\alpha = 1$ itself. So 1 is the only $\alpha$ such that $f_c(\alpha) = 0$ for multiple $f_c \in C(n,s)$.
\end{proof}

\begin{lem} \label{gcd_one}
Let $G \subseteq \F_q^*$  be the unique multiplicative subgroup of order $N$, and let $f \in \ff{q}[x]$ be a trinomial of the form $ax^{n} + bx^{s} + 1$ satisfying $\gcd(n, s, N) = 1$. 
The number of roots of $f$ that lie in $G$ does not exceed 
$$\frac{1}{2}+\sqrt{N}.$$
\end{lem}

\begin{proof}
Suppose $f(x) = ax^{n} + bx^{s} + 1$ has $r$ distinct roots $\zeta_1,\zeta_2, \ldots, \zeta_r$ in $G$.
For each $\zeta_i$ let $g_i(x) = f(\zeta_i x) = (a\zeta_i^n)x^{n} + (b\zeta_i^s)x^{s} + 1$.
Since the map $x \rightarrow \zeta x$ permutes the elements of $G$, each of these $g_i$ also has $r$ roots in $G$.
Additionally, each $g_i$ is a member of $C(n,s)$, since $g_i(1) = f(\zeta_i) = 0$.

We now check that each $g_i$ is distinct.
Suppose $g_i = g_j$ with $i \neq j$. We then have both $\zeta_i^n = \zeta_j^n$ and $\zeta_i^s = \zeta_j^s$, or, equivalently, $(\zeta_i/\zeta_j)^n = 1$ and $(\zeta_i/ \zeta_j)^s = 1$.
Once again, the only $\alpha \in G$ that satisfies $\alpha^n = \alpha^s = 1$ is $\alpha = 1$, so $\zeta_i = \zeta_j$, which contradicts the supposition that the roots $\zeta_1,\zeta_2, \ldots, \zeta_r$ are distinct.

In summary, there exist $r$ distinct trinomials of $C(n,s)$ that each have $r$ roots in $G$, and by Lemma \ref{disjoint_roots}, $\zeta = 1$ is the only root among these that is not unique.
This implies that $G$ contains at least $r(r-1) + 1$ distinct elements, but we know that $G$ has size $N$ by hypothesis.
Therefore it must be that
$$r^2 - r + 1 \leq N,$$ 
which yields the desired constraint on $r$:
$$r\leq \frac{1}{2}+\sqrt{N}.$$
\end{proof}
\noindent
We now have everything we need to complete the proof.

\begin{proof}[Proof of Theorem \ref{upper_bound}]
Let $f(x) = x^n + ax^s + b \in \F_q[x]$.
Obviously the roots of $f$ are not affected by re-scaling; let
$$ \tilde{f}(x) = \frac{1}{b}x^n + \frac{a}{b}x^s + 1 = \alpha x^n + \beta x^s + 1. $$
The exponents may fail to satisfy $\delta = \gcd(n, s, q-1) = 1$.
However, $\tilde{f}(x) = 0$ is equivalent to the system 
$$\alpha y^{n/\delta} + \beta y^{s/\delta} + 1= 0$$ 
$$y = x^{\delta}.$$
The second equation is only solvable for $x$ when $y$ lies in the subgroup of order $(q-1)/\delta$. 
The first equation satisfies $\gcd(n/\delta, s/\delta, (q-1)/\delta) = 1$, so we can invoke Lemma \ref{gcd_one} and find that there are at most $ \Big\lfloor \frac{1}{2} + \sqrt{\frac{q-1}{\delta}} \Big\rfloor$ such $y$. 
Each of these $y$ then admits one coset of $\delta$ distinct solutions for $x$.
\end{proof}

\begin{proof}[Proof of Theorem \ref{june_trinomials}]
Observe that the function 
$$T(x) = x^{p^k} + x$$ 
is an $\ff{p^k}$-linear map from $\ff{p^{2k}}$ to $\ff{p^{2k}}$. 
$T$ is a binomial, so it is easy to show that it does have nonzero solutions and therefore has a null space of positive dimension.
Since $T$ is not the zero transformation, we conclude that it has a null space of dimension 1, and therefore that it has $p^k$ roots.

We see that $f(x) = T(x) - 2 = 0$ exactly when $T(x) = 2$. 
This has one obvious solution, $x=1$, so we conclude from the linearity of $T$ that it has as many solutions as $T(x) = 0$.
Therefore, $f$ has $p^k = \sqrt{q}$ roots, all of which are nonzero.
\end{proof}

\section{Proof of Theorem \ref{poisson_density}}

Our proof of Theorem \ref{poisson_density} relies on the following statement from \cite{ffcheb}, which can be viewed as a Chebotarev density theorem for function fields. At its core, this result is powered by the Lang-Weil estimate for the number of points on varieties over $\F_q$. 

\begin{thm} \label{ffc}
\cite[Proposition 3.1]{ffcheb}
Let $n, m,$ and $N$ be positive integers, and let $F \in \F_q[A_1, \ldots A_m, x]$ be separable in $x$ and have $\deg F \leq N$ and $\deg_x F = n$.
Let $\F$ be an algebraic closure of $\F_q$, and suppose that 
$$\textnormal{Gal}\left(F , \F(A_1,\ldots A_m) \right) \cong S_n.$$
For a partition $\lambda$ of $n$, let $C_\lambda \subset S_n$ denote the conjugacy class of permutations $\sigma \in S_n$ with cycle type $\lambda$,
and let $\mathcal{A}_\lambda$ denote the set of $(a_1, \ldots, a_m) \in \F_q^m$ such that the univariate polynomial $f(x) = F(a_1, \ldots a_m, x)$ factorizes over $\F_q$ into irreducible factors with degree pattern $\lambda$.
Then, there exists a constant $c(m,N) \in \mathbb{R}$, which depends only on $m$ and $N$, such that
$$ \Big\lv \frac{\#\mathcal{A}_\lambda}{q^m} - \frac{\#C_\lambda}{\# S_n}  \Big\rv \leq \frac{c(m,N)}{ q^{1/2}}. $$
\end{thm}

Let $k$ be a field, and let $F(x) = x^n + A x^s + B$, where $A$ and $B$ are indeterminates over $k$, $0 < s < n$, and $\gcd(n,s) = 1$.
It is shown by Cohen in \cite[p.\ 64 and Corollary 3]{cohen} that unless char$(k)$ divides $n(n-1)$, $F$ is separable over $k(A,B)$ and $\textnormal{Gal}\left(F , k(A,B) \right) \cong S_n$. Here we consider $k$ an algebraic closure of a prime field $\F_p$, and $n$ bounded by some fixed $N \in \N$, so $F$ satisfies the conditions of Theorem \ref{ffc} when $p > N$.

Let $C(r)$ be the collection of all permutations $\sigma \in S_n$ with exactly $r$ fixed points. 
If $r = n$ then $C(r)$ contains only the identity permutation and then $ \frac{\#C(r)}{\#S_n} = 1/n! = 1/r! $.
Otherwise, every $\sigma \in C(r)$ can be written as $\sigma = c_1 c_2 \cdots c_r \sigma_d$, where each $c_i$ is a length-one cycle and $\sigma_d$ permutes the remaining elements and has no fixed points.
Permutations that have no fixed points are called \textit{derangements}, and the proportion of permutations that are derangements is extremely well-approximated by $e^{-1}$ \cite{derang}.
Specifically, the number of derangements of $n$ elements is given by $d_n =  \big[ e^{-1} n! \big],$
where $[\cdot]$ denotes the ``nearest integer" function.

Therefore, to count the the number of $\sigma \in C(r)$, we simply count the ways to choose $c_1, c_2, \ldots, c_r$ and multiply by the number of derangements of the remaining $n - r$ elements, so we have
$$
\frac{\#C(r)}{\#S_n} = \frac{{{n}\choose{r}}d_{n-r}}{n!} = \frac{\frac{n!}{r!(n-r)!}d_{n-r}}{n!} = \frac{\big[ e^{-1}(n-r)! \big]}{r!(n-r)!}.
$$
Note that
$$\frac{\big[ e^{-1}(n-r)! \big]}{r!(n-r)!} \leq \frac{ e^{-1}(n-r)! + 0.5}{r!(n-r)!} \leq \frac{ e^{-1}+ 0.5}{r!} < \frac{1}{r!},$$ 
so in fact we have $\frac{\#C(r)}{\#S_n} \leq 1/r!$ always.

Let $\mathcal{A}(r)$ denote the number of $(a,b) \in \F_p^2$ such that $F(a,b,x) = x^n + ax^s + b \in \F_p[x]$ has exactly $r$ linear factors. 
Since $C(r)$ is the union of some number of conjugacy classes which is bounded in terms of $N$, we have
$$ \Big\lv \frac{\#\mathcal{A}(r)}{p^2} - \frac{\#C(r)}{\# S_n} \Big\rv \leq  
\sum_{C_\lambda \subseteq C(r)} \Big\lv \frac{\#\mathcal{A}_\lambda}{p^2} - \frac{\#C_\lambda}{\# S_n}   \Big\rv 
= O_N \left( \frac{1}{ p^{1/2}} \right), $$ 
as $p \to \infty$, where the $O$-constant depends only on $N$. Now define 
$$\mathcal{A}^*(r) = \set{(a,b) \in (\F_p^*)^2}{x^n + a x^s + b \textnormal{ has exactly } r \textnormal{ distinct linear factors}}.$$
$\mathcal{A}^*(r)$ differs negligibly from $\mathcal{A}(r)$ since there are less than $2p$ elements in $\F_p^2 \setminus (\F_p^*)^2$, and by \cite[Proof of Proposition 3.1]{ffcheb}, the number of $(a,b) \in \F_p^2$ such that
$x^n + ax^s + b$ has a root of multiplicity is bounded asymptotically by $O_N(p)$, so
$$
\Big\lv \frac{ \#\mathcal{A}^*(r)}{ p^2}  - \frac{ \#\mathcal{A}(r)}{ p^2}   \Big\rv  
 = O_N\left( \frac{1}{p} \right) .
$$
Finally, note that 
$$ \Big\lv \frac{ \#\mathcal{A}^*(r)}{ (p-1)^2}  - \frac{ \#\mathcal{A}^*(r)}{ p^2}   \Big\rv  \leq 1 -  \frac{(p-1)^2}{p^2} < \frac{2}{p}.$$
Therefore we have 
$$\Big\lv \frac{\#\mathcal{A}^*(r)}{\#(\F_p^*)^2}  -  \frac{\#C(r)}{\#S_n}  \Big\rv = O_N \left(\frac{1}{p^{1/2}}\right), $$
which proves the first claim, concerning trinomials with $\gcd(n,s) = 1$.\\

Recall the definitions
\begin{align*}
M_p &= \set{f \in \F_p[x]}{ \deg f < p}  \\
T_{p,N} &= \set{x^n + ax^s + b}{(a,b) \in (\F_p^*)^2, \; 0 < s < n \leq N, \; \textnormal{and } \gcd(n,s,p-1) = 1},
\end{align*}
and recall that $\mu(p,r)$ denotes the proportion of $f \in M_p$ with $R(f) = r$, and  that $t_N(p,r)$ denotes the proportion of $f \in T_{p,N}$  with $R(f) = r$. It is clear that $t_N(p,r)$ is equal to the average value across all fractions
$$ \frac{\#\mathcal{A}^*(r)}{\#(\F_p^*)^2} $$
which are associated to a trinomial $x^n + Ax^s + B$ with $0 < s < n \leq N$ and $\gcd(n,s,p-1)= 1$.
It remains to study trinomials with $\gcd(n,s,p-1) = 1$ but $\gcd(n,s) > 1$.

\pagebreak

Suppose $k = \gcd(n,s) > 1$, and write $n = kn'$ and $s = ks'$ so that $\gcd(n',s') = 1$.
If $\gcd(n,s,p-1) = 1$, then we must have $\gcd(k,p-1) = 1$, so the map $x \rightarrow x^k$ permutes $\F_p$. \linebreak
Since $x^n = (x^k)^{n'}$ and $x^s = (x^k)^{s'}$, it follows that the trinomials $x^n + ax^s + b$ and $x^{n'} + ax^{s'} + b$
have the same number of distinct roots in $\F_p$. Thus, $t_N(p,r)$ is equal to the average of a collection of fractions which all satisfy
$$ \frac{\#\mathcal{A}^*(r)}{\#(\F_p^*)^2} \leq \frac{1}{r!} + \frac{C_N}{p^{1/2}},$$
where $C_N \in \mathbb{R}$ is a constant which depends only on $N$. It follows immediately that

$$ \limsup_{p \to \infty} t_N(p,r) \leq 1/r! $$
for each $r \leq N$, and so
$$ \limsup_{p \to \infty} \left(  \max_{r \leq N} t_N(p,r) r! \right) \leq 1 .$$

In \cite{random}, Leont'ev studies the generating function $\phi(x) = \sum_{r=0}^\infty \mu(p,r) x^r$, and shows that $\phi(x)$ converges to $e^{x-1}$ for $x \in (0,1]$.
Using the continuity theorem for generating functions \cite[Section 1.1.6]{cont_thm}, he then concludes that $\mu(p,r) \to e^{-1}/r!$ as $p \to \infty$ for all $r \in \N$.
Since we are only interested in the finitely many $r \in \{0,1,\ldots,N\}$, we can also be assured that
$$ \lim_{p \to \infty} \left( \min_{r \leq N} \mu(p,r) r! \right) = e^{-1}. $$
Therefore, we have
\begin{align*}
\limsup_{p \to \infty} \left(  \max_{r \leq N} \frac{t_N(p,r)}{\mu(p,r)} \right) &= 
\limsup_{p \to \infty} \left(  \max_{r \leq N} \frac{t_N(p,r) }{\mu(p,r) } \frac{r!}{r!} \right) \\
&\leq \limsup_{p \to \infty} \left( \frac{   \max_{r \leq N}  t_N(p,r)r! }{  \min_{r \leq N} \mu(p,r)r! } \right) \\
&= \frac{\limsup_{p \to \infty} \left(  \max_{r \leq N} t_N(p,r) r! \right)}{\lim_{p \to \infty} \left( \min_{r \leq N} \mu(p,r) r! \right)} \\
&\leq \frac{1}{e^{-1}} \\
&= e.
\end{align*}
\qed

\section{Poisson Heuristic and Computational Data for $\F_p$}

First, we attempt to establish some basic plausibility for the Poisson Heuristic.
As before, let $t(p,r)$ denote the proportion of trinomials over $\F_p$ with $\delta = 1$ that have $r$ distinct roots. 
The following table gives the statistical distance between $t$ and a Poisson distribution with mean $1$ for a few fields of various sizes.
\\
\begin{center}
\begin{tabular}{l|c}
	 $\F_p$ & $\sum_{r = 0}^{\infty} \lvert t(p,r) - e^{-1}/r! \rvert$   \\ \hline
	$\F_{101}$ & 0.0367266    \\ 
	$\F_{1009}$ & 0.0112061  \\ 
	$\F_{10007}$ & 0.0007107   \\ 
	$\F_{100003}$ & 0.0000834   \\ 
\end{tabular}
\captionof{table}{Deviation of $t(p,r)$ from a Poisson distribution.} 
\end{center}
\vspace{0.5cm}

Recall that $T_p$ denotes the set of trinomials over $\F_p$ with $\delta = 1$ and degree less than $p - 1$.
We have computed $R_p$, the maximum number of roots of attained by any $f \in T_p$, for primes up to $p = 139571$.
In this section, we show that the values of $R_p$ that we would expect by Heuristic \ref{behave} are quite close to what we actually observe.
That is, we consider the expected values of
$$R_p = \max \set{R(f)}{f \in T_p}$$ 
under the model that the values of $R(f)$ are given by random variables with distribution function $\rho(r) = {e^{-1}/r!}$,
and we compare these expected values with real values of $R_p$.

More generally, let $M_N$ be the maximum of $N$ independent variables all with distribution $\rho(r) = {e^{-1}/r!}$. 
It is known that $M_N$ becomes very predictable when $N$ is large.
Specifically, it is shown in \cite{poisson_exist} that there exists an integer sequence $\widehat{M}_N$ such that, as $N \to \infty$,
$$\textnormal{Prob}(\lvert \widehat{M}_N - M_N \rvert \leq 1) \to 1.$$
In \cite{poisson_weak}, a nice asymptotic formula is given for $\widehat{M}_N$:
$$\widehat{M}_N \sim \frac{\log{N}}{\log{\log{N}}}. $$

As an initial estimate, there are slightly less than $p^4$  trinomials $x^n + ax^s + b \in T_p$: there are $(p-1)^2$ pairs $(a, b)$ and almost $(p-1)^2$ pairs $(n,s)$. So, assuming Heuristic \ref{behave}, a reasonable conservative prediction would be
$$ R_p \approx \frac{4 \log{p}}{\log{\log{p}}}. $$

However, to make an accurate prediction for $R_p$ we need to be more precise in two ways.
Firstly, there are actually much fewer independent values of $R(f)$ than $p^4$.
For any $f \in \F_p[x]$, we have that 
$$R(f(x)) = R( f(\gamma x^e)),$$ 
as long as $\gcd(e, p - 1) = 1$ and $\gamma \in \F_p^*$,
because the maps $x \rightarrow \gamma x$ and $x \rightarrow x^e$ are both bijections on $\F_p$.
As a result, knowing the number of roots of one trinomial immediately determines the number of roots of a significant chunk of trinomials.
Therefore, we would like to find an appropriate, effective value for $N$ that better models the number of independent random values.
To do this, we count exactly the number of trinomials with $\delta = 1$ and then quotient out by the size of these equivalent chunks.

The exact number of pairs $(n, s)$ that are relatively prime with $p - 1$ is given by the \textit{Jordan totient function}, $J_2(p - 1)$ \cite[p.\ 147]{jordan}. 
We must subtract $\varphi(p-1)$ to avoid counting pairs with $n = s$, and we divide by $2$ to avoid counting both $(n,s)$ and $(s,n)$.
There are $(p-1)^2$ choices for the two coefficients, so overall we have $$\#T_p = \left( 1/2 \right) \left(p-1\right)^2 \left(J_2(p - 1) - \varphi(p - 1) \right) .$$
As discussed in Section 2, $\gcd(n, s, p - 1) = 1$ implies that every pair $(\gamma^n, \gamma^s)$ is unique, so we divide by $(p - 1)$ to account for trinomials of the form $f(\gamma x)$.
To account for the transformation $x \rightarrow x^e$, we divide by the number of $e$ with $\gcd(e,p-1) = 1$, which is given by $\varphi(p - 1)$.
So, we take our effective number of independent Poisson variables to be
$$ N(p) = \left(\frac{p-1}{2} \right)  \left( \frac{J_2(p-1)}{\varphi(p - 1)} - 1 \right).$$
This number is approximately equal to $p^2$; for primes in the range $11 \leq p \leq 139571$, we have
$$ \frac{1}{2} < \frac{N(p)}{p^2} < 2.$$

Secondly, it is beneficial to consider the less elegant but more precise asymptotic formula for $\widehat{M}_N$ given in \cite{poisson_strong}.
Below, $W$ is the \textit{Lambert W function}.
$$
\widehat{M}_N \sim E_N :=   \frac{\log N}{W(\log (N) / e)} - \frac{1 + \log 2\pi }{2 \log \left( \frac{\log N}{W(\log (N) / e)}\right) } - 1.5.\\
$$

In summary, by Heuristic \ref{behave} we expect that
$ R_p  \approx E_{N(p)} $
when $p$ is sufficiently large.
The following plot displays the ratios $R_p / E_{N(p)}$ for all primes $p \leq 139571$.

\noindent
\includegraphics[width=\textwidth]{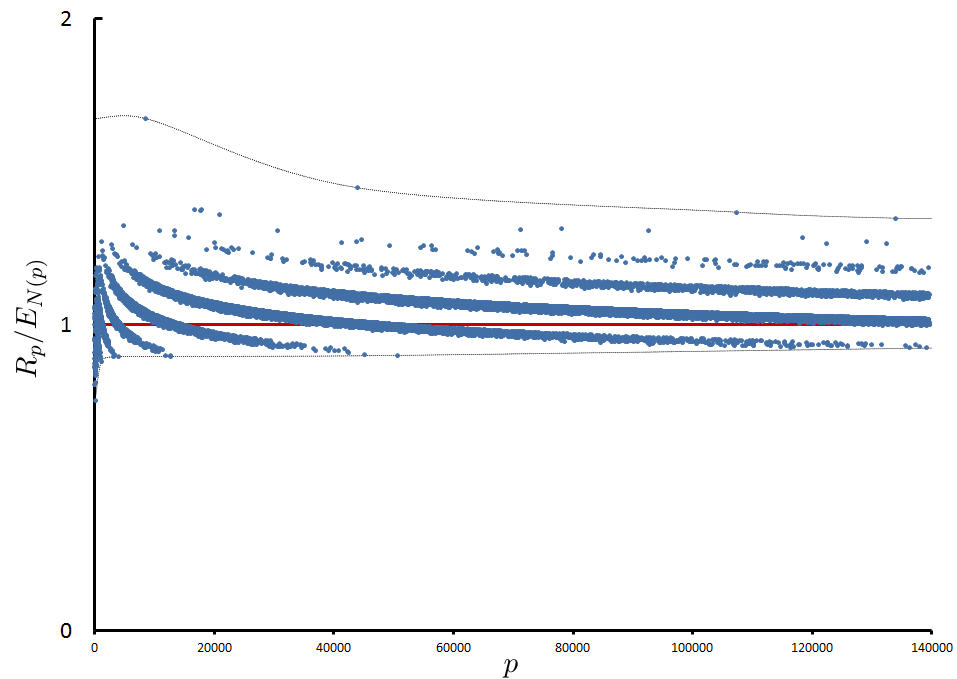}
\captionof{figure}{The ratios $R_p / E_{N(p)}$ for all primes $p \leq 139571$.}

\vspace{0.5cm}

The visibly distinct bands correspond to primes that share the same value for $R_p$. 
The apparent upper and lower bounding monotonic subsequences are traced by dotted curves.
The average over all ratios is 1.0429 and the standard deviation is 0.05587.
For all $p \leq 139571$, we have $R_p \leq 2 \log p$.
The largest recorded value of $R_p$ is $R_p = 16$, which is witnessed at $p = 8581, 43943, 107351,$ and $133877$; the associated ratios $16/E_{N(p)}$ lie visibly on the upper dotted line.

The values of $R_p = \max \set{R(f)}{f \in T_p}$ were computed in a straightforward way (i.e.\ by enumerating trinomials and counting their roots) 
by parallel C++ code which ran on Texas A\&M's Ada supercomputing cluster for $5000$ CPU hours.
The program takes advantage of the fact that $R(f(x)) = R(f(\gamma x^e))$ when $\gcd(e,p-1) = 1$ and $\gamma \in \F_p^*$ to reduce the enumeration space.
The values $E_{N(p)}$ were computed separately by a small Matlab program, which in particular makes use of Matlab's built-in \verb|lambertw| function.


\section*{Acknowledgments}
We would like to thank the Texas A\&M Supercomputing Facility for providing us 
with computational resources, and our advisor,  J.\ Maurice Rojas, for his 
indispensable guidance.

\end{document}